\documentclass{amsart}
\usepackage{docmute}

\usepackage{graphicx}
\usepackage{latexsym}
\usepackage{color}
\usepackage{amscd}
\usepackage[all]{xy}
\usepackage{enumerate}
\usepackage{subfigure}
\usepackage{soul}
\usepackage{comment}
\newtheorem{dfn}{Definition}[section]
\newtheorem{thm}[dfn]{Theorem}

\newtheorem{lem}[dfn]{Lemma}

\newtheorem{cor}[dfn]{Corollary}

\def\CM{\mathcal M}
\def\CN{\mathcal N}
\def\PM{\mathcal {PM}}
\def\PN{\mathcal {PN}}

\begin{document}

\title[Generating mapping class groups by involutions]{Generating the mapping class group of a non-orientable punctured surface by
    involutions}

\author{Kazuya Yoshihara}
\date{}

\begin{abstract}
    Let $N_{g,n}$ denote the closed non-orientable surface of genus $g$ with $n$ punctures and let $\CN_{g,n}$ denote the mapping class group of $N_{g,n}$.

    Szepietowski showed that $\CN_{g,n}$ is generated by finitely many involutions. The number of elements in his generating set depends linearly on $g$ and $n$. In the case of $n=0$, Szepietowski found an involution generating set in such a way that the number of its elements does not depend on $g$, showing that $\CN_{g,0}$ is generated by four involutions.
    In this thesis, for $n \geq 0$, we prove that $\CN_{g,n}$ is generated by eight involutions if $g \geq 13$ is odd and by eleven involutions if $g \geq 14$ is even.

\end{abstract}
\maketitle

\section{Introduction}
\label{section:one}
For $n \geq 0$, let $\Sigma _{g,n}$ (resp.\ $N_{g,n}$) denote the closed connected orientable (resp. non-orientable) surface of genus $g$ with arbitrarily chosen $n$ distinct points which we call {\it{punctures}}. The {\it{mapping class group}} $\CM_{g,n}$ (resp. $\CN_{g,n}$) is the group of isotopy classes of orientation preserving diffeomorphisms (resp.\ diffeomorphisms) of $\Sigma _{g,n}$ (resp.\ $N_{g,n}$) which preserve the set of punctures. Denote by $\PM_{g,n}$ (resp.\ $\PN_{g,n}$) the subgroup of $ \CM _{g,n}$ (resp.\ $ \CN_{g,n}$) consisting of the isotopy classes of diffeomorphisms which fix each puncture.

In the orientable case, Dehn \cite{dehn} and Lickorish \cite{li1} first proved that $\CM_{g,0}$ is generated by Dehn twists. Lickorish \cite{li2} showed that certain $3g-1$ Dehn twists generate $\CM_{g,0}$ for $g \geq 1$. This number was improved to be $2g+1$ by Humphries \cite{hum} for $g \geq 3$. Moreover, Humphries showed that $\CM_{g,0}$ cannot be generated by $2g$ (or less) Dehn twists for any $g \geq 2$.

It has been extensively studied the problem of finding torsion generators for finite groups and mapping class groups. This study for $\CM_{g,n}$ was started by Maclachlan \cite{mac}. He proved that $\CM_{g,0}$ is generated by torsion elements and used this result to show that the moduli space of Riemann surfaces of genus $g$ is simply connected as a topology space. Patterson \cite{patter} showed that $\CM_{g,n}$ is generated by torsion elements for $g \geq 3$ and $n\geq1$. Korkmaz \cite{ko2} showed that $\CM_{g,n}$ is generated by two elements of order $4g+2$ for $g \geq 3$ and $n=0,1$. McCarthy and Papadopoulos \cite{mc_pa} proved that $\CM_{g,0}$ is generated by infinitely many conjugates of a certain involution. Luo \cite{luo} showed that $\CM_{g,n}$ is generated by $12g+1$ involutions for $g \geq 3,n \leq1$. In his paper, Luo asked the following question:
\textit{Is there a universal upper bound which is independent of $g$ and $n$ for the number of torsion elements necessary to generate $\CM_{g,n}$}?
Brendle and Farb \cite{b_f} gave a positive answer to Luo's question for $n=0,1$. They found a generating set for $\CM_{g,n}$ which consists of six involutions. Moreover, they showed that $\CM_{g,n}$ can be realized as a quotient of a Coxeter group on six generators. For every $n \geq 0$, Kassabov \cite{kass} proved that $\CM_{g,n}$ is generated by four (resp.\ five or six) involutions if $g \geq 8$ (resp.\ if $g \geq 6$ or if $g \geq 4$). Monden \cite{mon1} proved that $\CM_{g,n}$ is generated by four (resp.\ five) involutions if $g \geq 7$ (resp.\ if $g \geq 5$). Korkmaz\cite{ko3} improved this result for $n=0$. He proved that $\CM_{g,0}$ is generated by three (resp.\ four) involutions if $g \geq 8$ (resp.\ if $g \geq 3$). Yildiz\cite{y} proved that $\CM_{g,0}$ is generated by three involutions if $g \geq 6$.

In the non-orientable case, Lickorish \cite{li3} first proved that $\CN_{g,0}$ is generated by Dehn twists and Y-homeomorphisms. Chillingworth \cite{chil} found a finite set of generators of this group. Korkmaz \cite{ko1} found finite generating sets for the groups $\CN_{g,n}$ and $\PN_{g,n}$. The number of Chillingworth's generators is improved to $g+1$ by Szepietowski \cite{sez3}. Hirose \cite{hirose} proved that his generating set is the minimal generating set by Dehn twists and Y-homemorphisms.
Szepietowski \cite{sez1} proved that $\CN_{g,n}$ is generated by involutions. The cardinality of his generating set of involutions depends linearly on $g$ and $n$. We can consider Luo's problem for $\CN_{g,n}$: {\textit{Is there a universal upper bound which is independent of $g$ and $n$ for the number of torsion elements necessary to generate $\CN_{g,n}$}?}  In the case $n=0$, Szepietowski gave a positive answer and found four involutions which generate $\CN_{g,0}$ for $g \geq 4$ \cite{sez2}. Altunoz, Pamuk, and Yildiz\cite{a_p_y} proved that $\CN_{g,0}$ is generated by three involutions if $g \geq 26$. But, in the case $n \neq 0$, it is not known. We will gave a positive answer for this problem.

\begin{thm}
    \label{thm:main}
    Let $n$ be a non-negative integer.
    Then, for $g$ odd with $g \geq 13$, $\CN_{g,n}$ is generated by eight involutions.
    For $g$ even with $g \geq 14$, $\CN_{g,n}$ is generated by eleven involutions.
\end{thm}

The paper is organized as follows.
In Section $2$ we recall the properties of Dehn twists, Y-homeomorphisms and puncture slides. In Section $3$, we construct involutions of $\CN_{g,n}$ and prove the theorem~\ref{thm:main}.

\section{Preliminaries}
\label{section:two}
Let $N_{g,n}$ be the closed non-orientable surface of genus $g$ with $n$ punctures and let $\Delta$ be the set of punctures of $N_{g,n}$. We represent the surface $N_{g,n}$ as a connected sum of an orientable surface and one or two projective planes (one for $g$ odd and two for $g$ even). In Figures.~\ref{fig:surface_odd} and ~\ref{fig:surface_even}, each encircled cross mark represents a crosscap: the interior of the encircled disk is to be removed and each pair of antipodal points on the boundary are to be identified.

\begin{figure}[ht]
    \begin{center}
        \includegraphics[scale=1.0,clip]{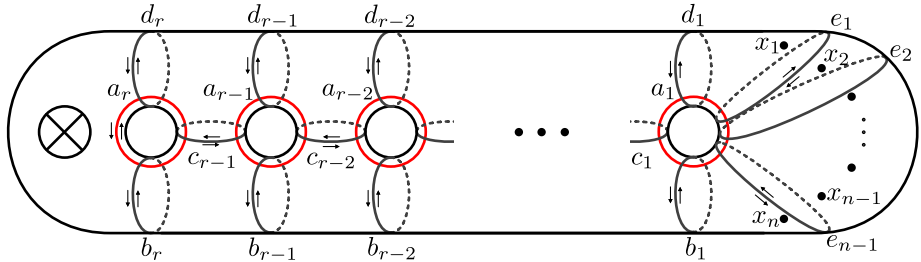}
    \end{center}
    \caption{Surface $N_{g,n}$ for $g=2r+1$ and its simple closed curves}
    \label{fig:surface_odd}
\end{figure}

\begin{figure}[ht]
    \begin{center}
        \includegraphics[scale=1.0,clip]{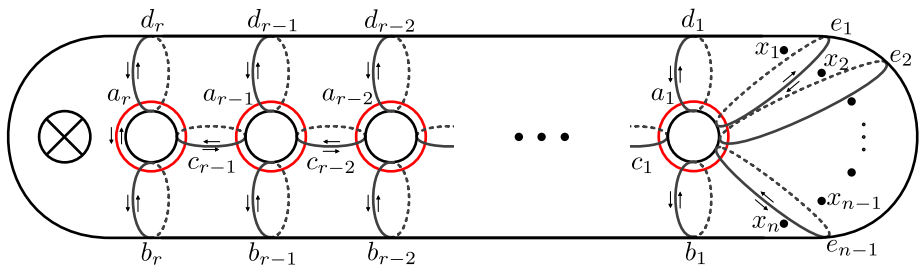}
    \end{center}
    \caption{Surface $N_{g,n}$ for $g=2r+2$ and its simple closed curves}
    \label{fig:surface_even}
\end{figure}

The set of all diffeomorphisms of $N_{g,n}$ which preserve the set of punctures obviously forms a group, which we denote by $\mathrm{Diff} (N_{g,n}) $. Let $\mathrm{Diff}_0 (N_{g,n})$ be the subset consisting of all elements of $\mathrm{Diff} (N_{g,n})$ that are isotopic to the identity, where the isotopies fix $\Delta$. It is immediately seen that $\mathrm{Diff}_0 (N_{g,n})$ is a normal subgroup of $\mathrm{Diff} (N_{g,n})$. The mapping class group of $N_{g,n}$, denoted by $\CN_{g,n}$, is the quotient group $\mathrm{Diff} (N_{g,n}) / \mathrm{Diff} _0 (N_{g,n})$. We denote by $\PN_{g,n}$ the subgroup of $\CN_{g,n}$ consisting of the isotopy classes of diffeomorphisms which fix each puncture. Let $Sym_n$ be a symmetric group on $n$ letters. Clearly we have the exact sequence
\[ 1 \to \PN_{g,n} \to \CN_{g,n} \stackrel{\pi}{\to} Sym_n \to 1, \]
where the last projection is given by the restriction of homeomorphism to its action on the puncture points.
Let $c$ be a simple closed curve on $N_{g,n}$.  If the regular neighborhood of $c$, denoted by $N(c)$, is an annulus (resp. a M\"obius band), we call $c$ \textit{two-sided} (resp.\ \textit{one-sided}) simple closed curve. Let $a$ be a two-sided simple closed curve on $N_{g,n}$. By the definition, the regular neighborhood of $a$ is an annulus, and it has two possible orientation. Now, we fix one of its two possible orientations. For two-sided simple closed curve $a$, we can define the Dehn twist $t_a$. We indicate the direction of a Dehn twist by an arrow beside the curve $a$ as shown in Figure.~\ref{fig:dehn_twist}.

\begin{figure}[ht]
    \begin{center}
        \includegraphics[scale=1.0,clip]{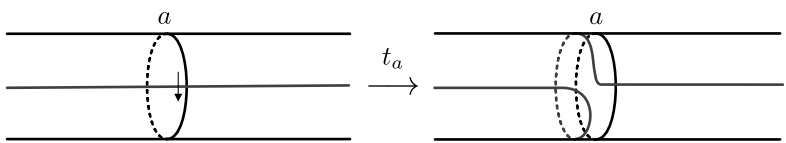}
    \end{center}
    \caption{Dehn twist along a simple closed curve $a$}
    \label{fig:dehn_twist}
\end{figure}

It is well known that $\CN_{g,n}$ is not generated by Dehn twists. We need another class of diffeomorphisms, called Y-homeomorphism, to generate $\CN_{g,n}$. A Y-homeomorphism is defined as follow. For a one-sided simple closed curve $m$ and a two-sided oriented simple closed curve $a$ which intersects $m$ transversely in one point, the regular neighborhood $K$ of $m \cup a$ is homomeomorphic to the Klein bottle with one hole. Let $M$ be the regular neighborhood of $m$. Then the {\textit{Y-homeomorphism}} $Y_{m,a}$ is the isotopy class of the diffeomorphism obtained by pushing $M$ once along $a$ keeping the boundary of $K$ fixed (see Figure.~\ref{fig:y-homeo}).\\

\begin{figure}[ht]
    \begin{center}
        \includegraphics[scale=0.8,clip]{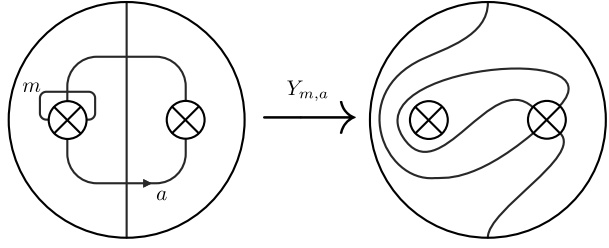}
    \end{center}
    \caption{Y-homeomorphism on $K$}
    \label{fig:y-homeo}
\end{figure}

Furthermore, to generate the groups $\CN_{g,n}$ and $\PN_{g,n}$ we need a puncture slide. A puncture slide is defined as follow. Let $M$ denote a M\"obius band with a puncture $x$ embedded in $N_{g,n}$. For a one-sided simple closed curve $\alpha$ based at $x$ on $M$, we push the puncture $x$ once along $\alpha$ keeping the boundary of $M$ fixed. Then a {\textit{puncture slide}} on $M$ is described as the result.

\begin{figure}[ht]
    \begin{center}
        \includegraphics[scale=0.8,clip]{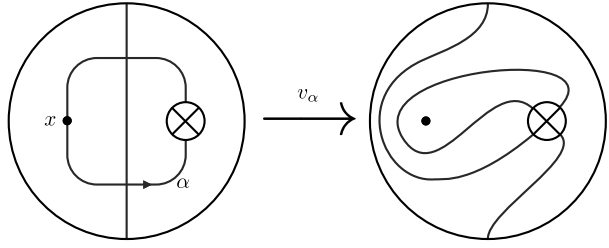}
    \end{center}
    \caption{Puncture slide on $M$}
    \label{fig:puncture_slide}
\end{figure}

These diffeomorphisms have the following properties.

\begin{lem}
    \label{lem:Property_Dehn_Twist}
    For any diffeomorphism $f$ of the surface $N_{g,n}$ and a two-sided simple closed curve $a$, we have

    \[ t_{f(a)}^{\epsilon} = f t_a f^{-1}, \]
    where if $f\mid_{N_a}$ is an orientation preserving diffeomorphism (resp.\ orientation reversing diffeomorphism), then $\epsilon = 1$ (resp. $\epsilon = -1$).
\end{lem}

\begin{lem}
    \label{lemma:Property_Y_homeo}
    For a one-sided simple closed curve $m$ and a two-sided simple closed curve $a$, we have the following.\\
    $(1)$ $Y_{m^{-1},a} = Y_{m,a}$. \\
    $(2)$ $Y_{m,a^{-1}} = Y_{m,a}^{-1}$. \\
    $(3)$ For any element $f$ in $\CN_{g,n}$, we have $f Y_{m,a} f^{-1} = Y_{f(m),f(a)}$. \\
\end{lem}

\begin{lem}
    \label{lemma:Property_puncture_slide}
    Let $v$ be a puncture slide of $x$ along a one-sided simple closed curve $\alpha$.
    For any element $f$ in $\CN_{g,n}$,
    $f v f^{-1}$ is the puncture slide of $f(x)$ along $f(\alpha)$.
\end{lem}

\section{Proof of Theorem~\ref{thm:main} }
\label{section:three}
In this section, we proof theorem~\ref{thm:main}. We use the following lemma in order to prove theorem~\ref{thm:main}.
\begin{lem}
    \label{lemma:main_tool_of_proof_main_theorem_two}
    Let $G$ and $N$ be groups and let $H$ and $K$ be subgroups of $G$. Suppose that the sequence
    \[ 1 \to H \stackrel{i}{\to} G \stackrel{\pi}{\to} N \to 1 \]
    is exact.
    If $K$ contains $i(H)$ and the restriction of $\pi$ to $K$ is a surjection onto $N$, then we have that $K = G$.
\end{lem}

\begin{proof}
    Let $g$ be any element of $G$. If $g$ is in $i(H)$, then $K$ contains $g$ by the assumption $i(H) \subset K$. We suppose that $g$ is not in $i(H)$. Since the restriction $\pi \mid _K$ is surjection, there exists $k \in K$ such that $\pi(g) = \pi(k)$. Since $\pi(g k^{-1}) = e$, we see that $g k^{-1} \in Ker~\pi = Im~i$. Therefore, there exists $h \in H$ such that $g k^{-1} = i(h)$. Since $i(h) \in K$, we have $g = i(h) k \in K$. Hence, $G \subset K$.
\end{proof}

Since we have the following exact sequence
\[ 1 \to \PM _{g,n} \to \CN _{g,n} \stackrel{\pi}{\to} \mathrm{Sym}_n \to 1, \]
we have following corollary.

\begin{cor}
    Let $K$ denote the subgroup of $\CN _{g,n}$. If $K$ contains $\PM _{g,n}$ and the restriction $\pi$ to $K$ is a surjection to $\mathrm{Sym}_n$, then $K$ is equal $\CN _{g,n}$.
\end{cor}

We recall the Korkmaz's generating set for $\PM _{g,n}$. Let $\Lambda$ be the set of simple closed curves indicated in Figure.~\ref{fig:surface_odd} for $g=2r+1$, and in Figure.~\ref{fig:surface_even} for $g=2r+2$. Hence
\[ \Lambda = \{a_1, {a_2}, \ldots, {a_r},{b_1}, {b_2}, \ldots, {b_r},{c_1},{c_2},\ldots,{c_{r-1}},{d_1},{d_2}, \ldots, {d_r},{e_1},{e_2},\ldots,{e_{n-1}} \} \]
for $g=2r+1$, and
\[ \Lambda = \{a_1, {a_2}, \ldots, {a_r},{b_1}, {b_2}, \ldots,{b_{r+1}},{c_1},{c_2},\ldots,{c_r},{d_1},{d_2}, \ldots, {d_r},{e_1},{e_2},\ldots,{e_{n-1}} \} \]
fro $g=2r+2$.
In the figures, we choose orientations of local neighborhoods of simple closed curves in $\Lambda$, the orientation is that the arrow points to the right if we approach the curve. Therefore for the simple closed curve $a$ in $\Lambda$, the Dehn twist about $a$ is determined by this particular choice of orientation.

Let $\alpha_i$ be the one-sided simple closed curve based at $x_i$ for $i=1,2,\ldots,n$ as in Figure.~\ref{fig:curves_alpha_and_beta}. If $g=2r+2$, let $\beta_i$ be the one-sided simple closed curve based at $x_i$ as in Figure.~\ref{fig:curves_alpha_and_beta}. For $i=1,2,\ldots,n$, let $v_i$ and $w_i$ be puncture slides along $\alpha_i$ and $\beta_i$, respectively.

\begin{figure}[ht]
    \begin{center}
        \includegraphics[scale=1.0,clip]{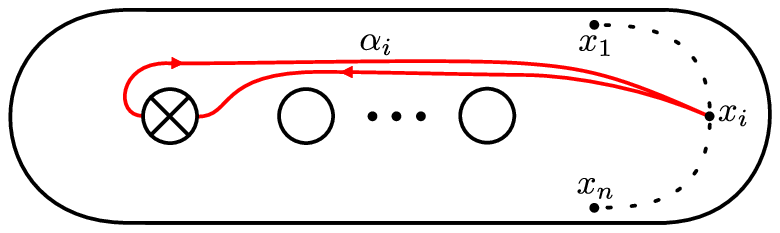}
        \vspace{1pt}
        \includegraphics[scale=1.0,clip]{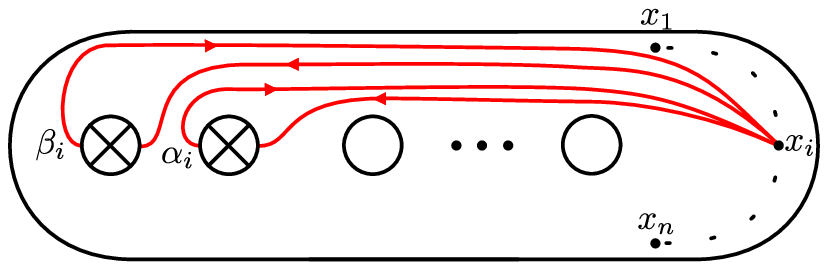}
    \end{center}
    \caption{Simple closed curves $\alpha_1,\ldots, \alpha_n$ and $\beta_1,\ldots,\beta_n$.}
    \label{fig:curves_alpha_and_beta}
\end{figure}

Let $y$ be a crosscap slide such that $y^2$ is the Dehn twist along $\xi$.

\begin{figure}[ht]
    \begin{center}
        \includegraphics[scale=0.6,clip]{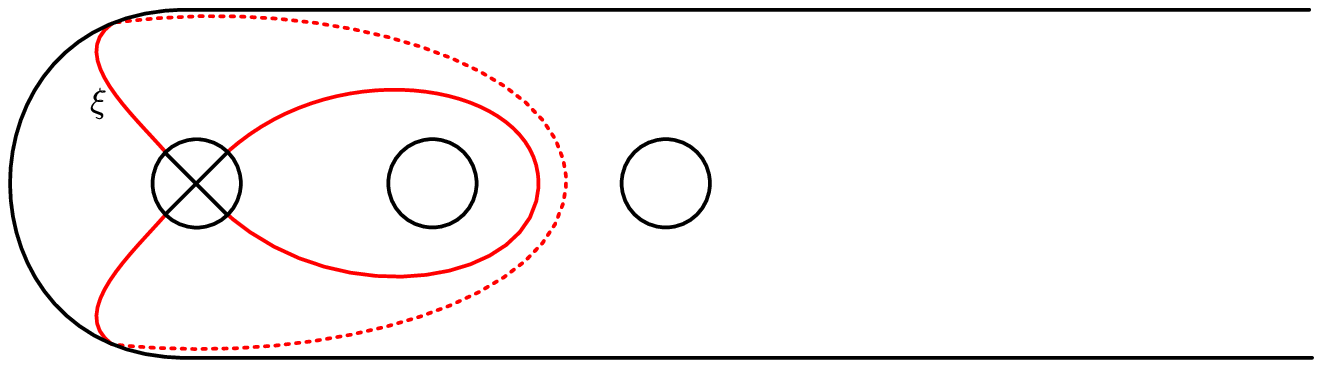}
        \hspace{1pt}
        \includegraphics[scale=0.6,clip]{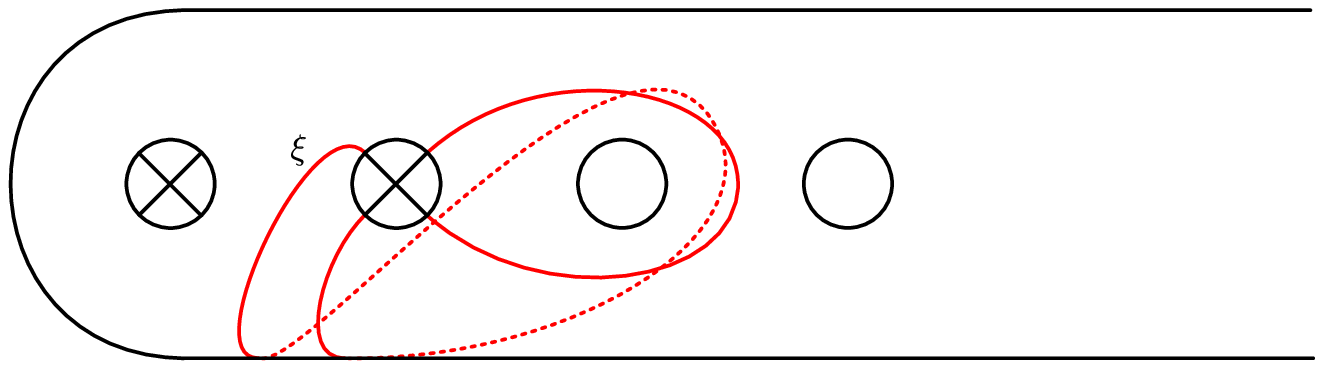}
    \end{center}
    \caption{Simple closed curve $\xi$.}
    \label{fig:simple_closed_curve_xi}
\end{figure}

\begin{thm}
    For $g\geq3$, the pure mapping class group $\PM _{g,n}$ is generated by
    \[(i)~ \{t_l \mid l \in \Lambda \} \cup \{v_i \mid 1\leq i \leq n\} \cup \{y\}~\mbox{if $g$ is odd, and} \]
    \[(ii)~ \{t_l \mid l \in \Lambda \} \cup \{v_i,w_i \mid 1\leq i \leq n\} \cup \{ y \} ~\mbox{if $g$ is even.} \]
\end{thm}

The following theorem can be deduced from Korkmaz's generating set by using the method of Humphries. Set
\[ \Lambda' = \{a_1, {a_2}, \ldots, {a_r},{b_1}, {b_2},{c_1},{c_2},\ldots,{c_{r-1}},{d_1},{d_2},{e_1},{e_2},\ldots,{e_{n-1}} \} \]
for $g=2r+1$, and
\[ \Lambda' = \{a_1, {a_2}, \ldots, {a_r},{b_1}, {b_2}, {b_{r+1}},{c_1},{c_2},\ldots,{c_r},{d_1},{d_2},{e_1},{e_2},\ldots,{e_{n-1}} \} \]
fro $g=2r+2$.

\begin{thm}
    For $g\geq3$, the pure mapping class group $\PM _{g,n}$ is generated by
    \[(i)~ \{t_l,v_i,y \mid l \in \Lambda',1\leq i \leq n\}~\mbox{if $g$ is odd and} \]
    \[(i)~ \{t_l,v_i,w_i,y \mid l \in \Lambda',1\leq i \leq n\}~\mbox{if $g$ is even.} \]
\end{thm}

\subsection{In the case of odd genus}
In this subsection, we suppose that $g=2r+1$ for a positive integer $r \geq 6$. Let us consider the two models of $N_{g,b}$ as shown in Figure.~\ref{fig:reflection_sigma} and \ref{fig:reflection_tau}. (In these pictures, we will suppose that $r=2k$ and the number of punctures $b=2l+1$ is odd for an integer $l \geq 0$.)
We deform the surface in Figure.~\ref{fig:reflection_sigma} from the surface in Figure.~\ref{fig:surface_odd} by diffeomorphism $\psi$ such that the simple closed curves and the punctures in Figure.~\ref{fig:surface_odd} map to the curves and punctures with same label in Figure.~\ref{fig:reflection_sigma}, and the deformed surface is symmetrical about a plane across the central of this surface, which we call mirror. Let $\sigma'$ be a reflection of this surface in the mirror and let $\sigma$ be a product $\psi^{-1} \sigma \psi$. Then $\sigma$ is involution in $\CN _{g,n}$.
In the same way, we can define a involution $\tau$ as a reflection in a mirror in Figure.~\ref{fig:reflection_tau}.

\begin{figure}[ht]
    \begin{center}
        \includegraphics[scale=0.90,clip]{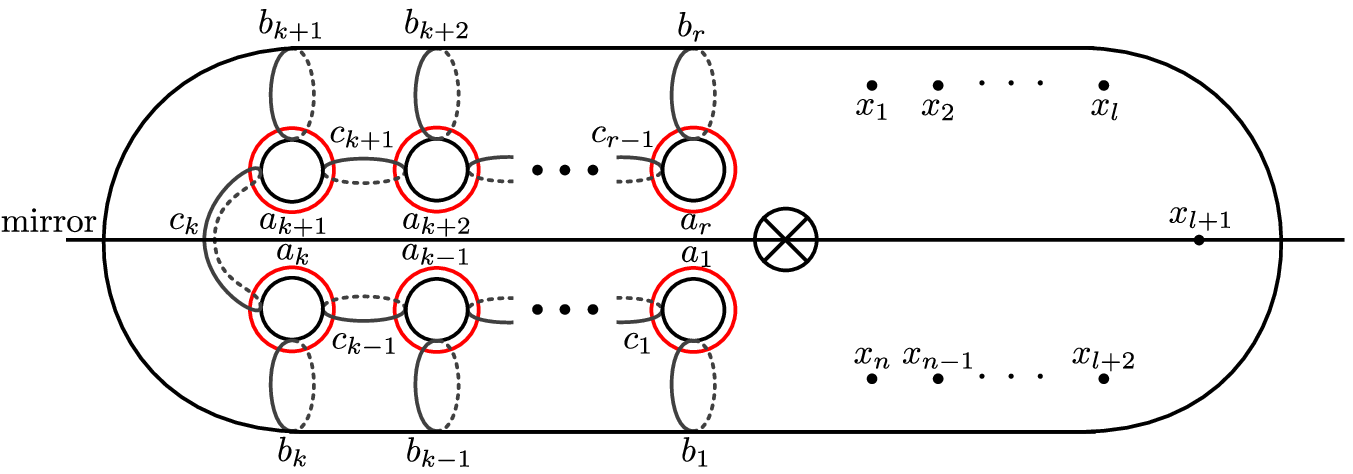}
    \end{center}
    \caption{Involution $\sigma : N_{g,n} \to N_{g,n}$}
    \label{fig:reflection_sigma}
\end{figure}

\begin{figure}[ht]
    \begin{center}
        \includegraphics[scale=0.9,clip]{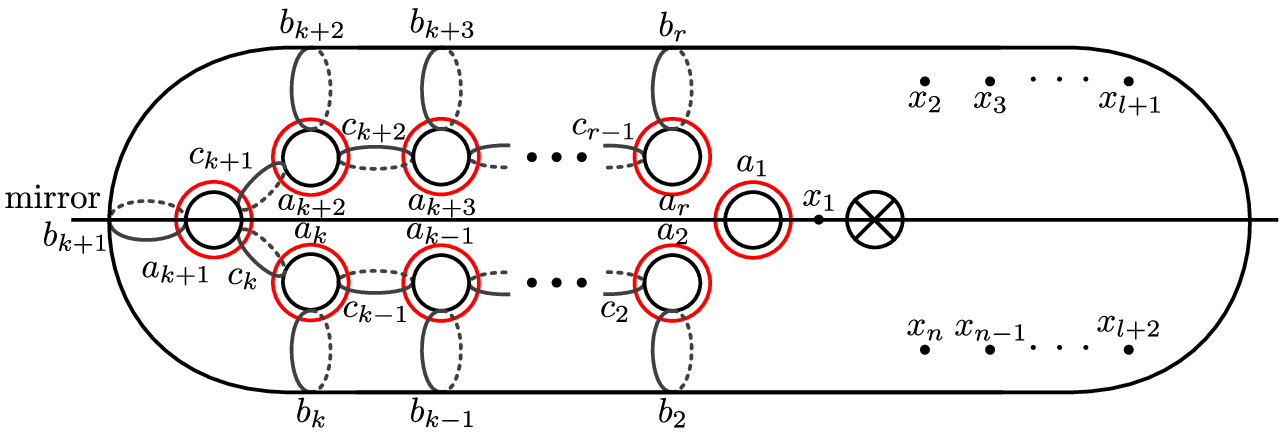}
    \end{center}
    \caption{Involution $\tau : N_{g,n} \to N_{g,n}$}
    \label{fig:reflection_tau}
\end{figure}

We will construct the third involution $I$. We cut the surface $N_{g,n}$ along $a_{k+3} \cup b_k \cup c_k \cup c_{k+1} \cup x$ to obtain the surfaces $S_1$ and $S_2$.(see Figure.\ref{fig:cut_curves_for_I_odd_case}) $S_1$ is a sphere bounded by $a_{k+3} \cup b_k \cup c_k \cup c_{k+1} \cup x$ and $S_2$ is a non-orientable surface of genus $g-8$ with $b$ punctures and $5$ boundaries. Figure.~\ref{fig:involution_I_odd_case} gives the involutions $\overline{I}$ and $\widetilde{I}$ on $S_1$ and $S_2$, respectively. Since $\overline{I}$ and $\widetilde{I}$ coincide on the boundaries, they naturally define an involution $I:N_{g,n} \to N_{g,n}$.

\begin{figure}[ht]
    \begin{center}
        \includegraphics[scale=1.0,clip]{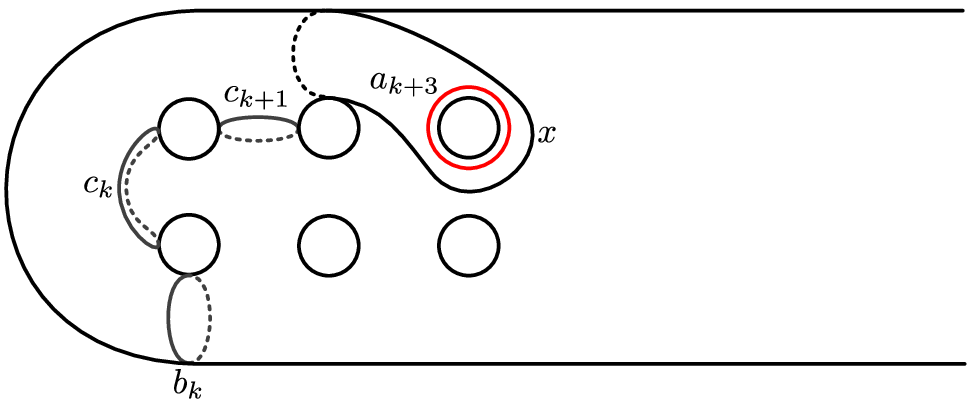}
    \end{center}
    \caption{The curves $a_{k+3}, b_k, c_k, c_{k+1}$ and $x$}
    \label{fig:cut_curves_for_I_odd_case}
\end{figure}

\begin{figure}[ht]
    \begin{center}
        \includegraphics[scale=0.9,clip]{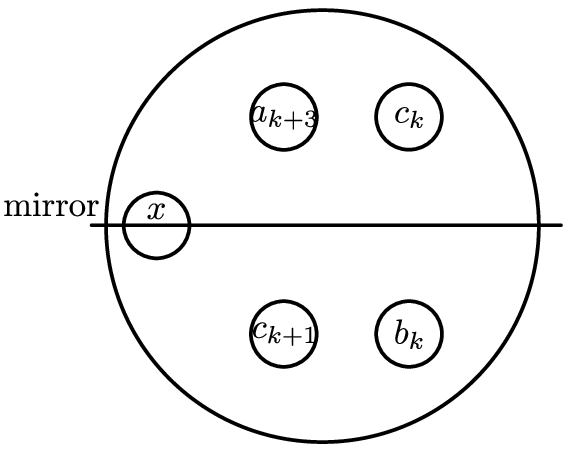}
        \hspace{0.5cm}
        \includegraphics[scale=0.9,clip]{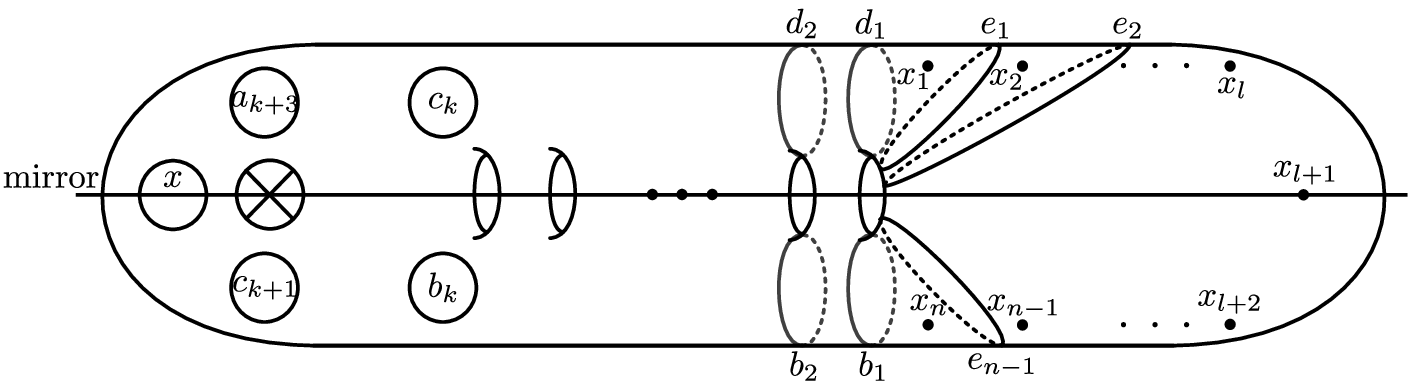}
    \end{center}
    \caption{Involutions $\overline{I}$ and $\widetilde{I}$}
    \label{fig:involution_I_odd_case}
\end{figure}

From the construction of $I$, we see the following:
\[ I (a_{k+3}) = c_{k+1} , I(c_k) = b_k, \]
\[ I (b_1) = d_1 , I (b_2) = d_2. \]

Let $\rho _1$ be the product $\tau t_{a_1}$. Since $\tau$ fixes $a_1$ and the restriction $\tau \mid _{N({a_1})}$ reverses the orientation, by Lemma \ref{lem:Property_Dehn_Twist}, we see that
\[ \tau t_{a_1} \tau = t_{a_1} ^{-1}.\]
Hence, $\tau$ is an involution.
Then we can get following lemma.

\begin{lem}
    \label{lem:dehn_twist_generate1}
    Dehn twists $t_{a_1},t_{a_2},\ldots,t_{a_r}, t_{b_1},t_{b_2}$, $t_{c_1},t_{c_2},\ldots,t_{c_{r-1}}$,$t_{d_1}$ and $t_{d_2}$ are products of involutions $\sigma$,$\tau$,$\rho_1$ and $I$.
\end{lem}

\begin{proof}

    Let R be the product $\tau \sigma$. We can see that $R$ acts as following by Figure.~\ref{fig:reflection_sigma} and Figure.~\ref{fig:reflection_tau}.
    \[ (1) R(a_1) = a_2, R(a_2) = a_3, \ldots, R(a_k)=a_{k+1}, R(a_{k+1}) = a_{k+2}, \ldots, R(a_{r-1}) = a_r.\]
    \[ (2) R(b_1) = b_2, R(b_2) = b_3, \ldots, R(b_k)=b_{k+1}, R(b_{k+1}) = b_{k+2}, \ldots, R(b_{r-1}) = b_r.\]
    \[ (3) R(c_1) = c_2, R(c_2) = c_3, \ldots, R(c_k)=c_{k+1}, R(c_{k+1}) = c_{k+2}, \ldots, R(c_{r-2}) = c_{r-1}. \]
    Clearly, we can see that $t_{a_1}$ is a product of $\tau$ and $\rho_1$. By $(1)$ and Lemma~\ref{lem:Property_Dehn_Twist},
    \[ t_{a_i} = R t_{a_{i-1}} R^{-1}. (i=2,3,\ldots,r) \]
    So $t_{a_1},t_{a_2},\ldots,t_{a_r}$ are products of $\sigma,\tau$, and $\rho_1$.\\

    By construction of $I$ and Lemma~\ref{lem:Property_Dehn_Twist}, we have
    \[ t_{c_{k+1}} = I t_{a_{k+3}} ^{-1} I.\]
    By $(3)$ and Lemma \ref{lem:Property_Dehn_Twist}, we see that
    \[ t_{c_j} = R t_{c_{j-1}} R^{-1}, \hspace{4pt} (j=k+2,k+3,\ldots,r-1) \]
    \[ t_{c_j} = R^{-1} t_{c_{j+1}} R. \hspace{4pt} (j=1,2,\ldots,k) \]
    Hence, $t_{c_1},t_{c_2},\ldots,t_{c_{r-1}}$ are products of $\sigma,\tau,\rho_1$, and $I$.\\

    Also, we have
    \[ t_{{b}_k} = I t_{c_k}^{-1} I.\]
    Similar to the above, by $(2)$ and Lemma \ref{lem:Property_Dehn_Twist}, we see that
    \[ t_{b_i} = R t_{b_{i-1}} R^{-1}, \hspace{4pt} (i=k+1,k+3,\ldots,r) \]
    \[ t_{b_i} = R^{-1} t_{b_{i+1}} R. \hspace{4pt} (i=1,2,\ldots,k-1) \]
    Hence, $t_{b_1},t_{b_2},\ldots,t_{b_r}$ are product of $\sigma,\tau,\rho_1$, and $I$.\\

    Finally, Since $I(b_1) = d_1$ and $I(b_2) = d_2$, we have
    \[ t_{d_1} = I t_{b_1} ^{-1} I , t_{d_2} = I t_{b_2} ^{-1} I.\]
    $t_{d_1}$ and $t_{d_2}$ are products of $\sigma,\tau,\rho_1$, and $I$.
\end{proof}

$\tau$ maps $\alpha _1$ to itself but reverses the orientation of $\alpha _1$. By Lemma~\ref{lemma:Property_puncture_slide}, we see that
\[ \tau v_1 \tau = v_1^{-1}. \]
Now let $\rho_2$ denote a product of $\tau v_1$.Then $\rho_2$ is a involution.

\begin{figure}[ht]
    \begin{center}
        \includegraphics[scale=0.8,clip]{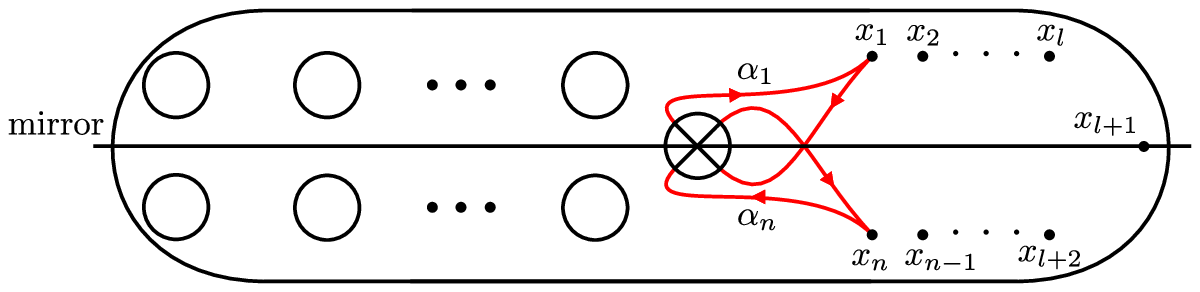}
        \hspace{0.3cm}
        \includegraphics[scale=0.8,clip]{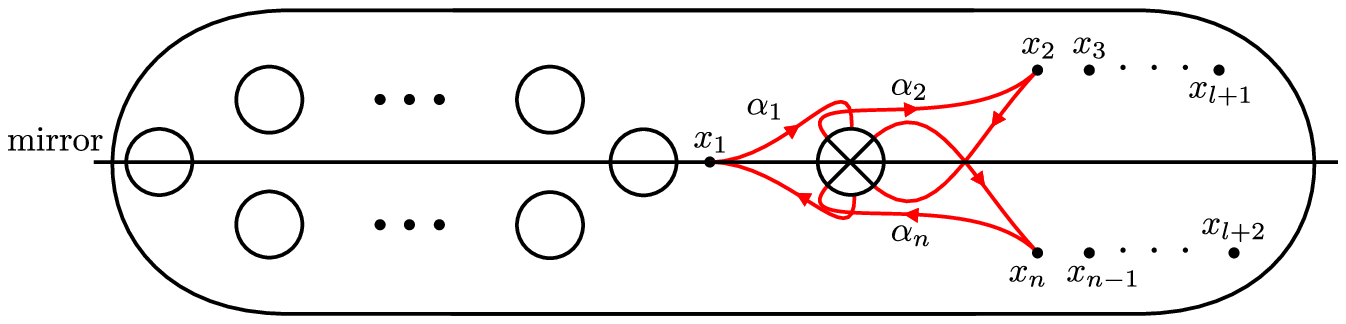}
    \end{center}
    \caption{Involutions $\sigma$ and $\tau$}
    \label{fig:involution_sigma_tau_other_view}
\end{figure}

\begin{lem}
    \label{lem:puncture_slide_generate}
    Puncture slides $v_i  (i=1,2,\ldots,n)$ is a products of involutions $\sigma,\tau$ and $\rho_2$.
\end{lem}

\begin{proof}
    $v_1$ is a product of $\tau$ and $\rho _2$. In Figure.~\ref{fig:involution_sigma_tau_other_view}, we fucus the figures which define $\sigma$ and $\tau$ on $\alpha _i$. $R = \tau \sigma$ acts on $\alpha _i$ as follow.
    \[(4)~ R(\alpha _1) = \alpha_2, R(\alpha_2) = \alpha_3, \ldots, R(\alpha_l)=\alpha_{l+1}, R(\alpha_{l+1}) = \alpha_{l+2}, \ldots, R(\alpha_{n-1}) = \alpha_n. \]

    By $(4)$ and Lemma~\ref{lemma:Property_puncture_slide}, we see that
    \[ v_j = R v_{j-1} R^{-1}. \hspace{4pt} (j=2,3,\ldots,n) \]
    Hence, $v_i$ is a product of involution $\sigma,\tau$ and $\rho_2$.
\end{proof}

We consider the diffeomorphism $\Phi$ on $N_{g,n}$ which satisfies $\Phi y \Phi^{-1} = Y_{m,a}$ and fixes each punctures. The bottom figure in Figure.~\ref{fig:involution_W} gives the involution $w$. Since $w$ fixes $m$ and $a$ but reverses the orientation of $m$ and $a$, we can see that $w Y_{m,a} w = Y_{m^{-1},a^{-1}} = Y_{m,a} ^{-1}$.

\begin{figure}[ht]
    \begin{center}
        \includegraphics[scale=1.0,clip]{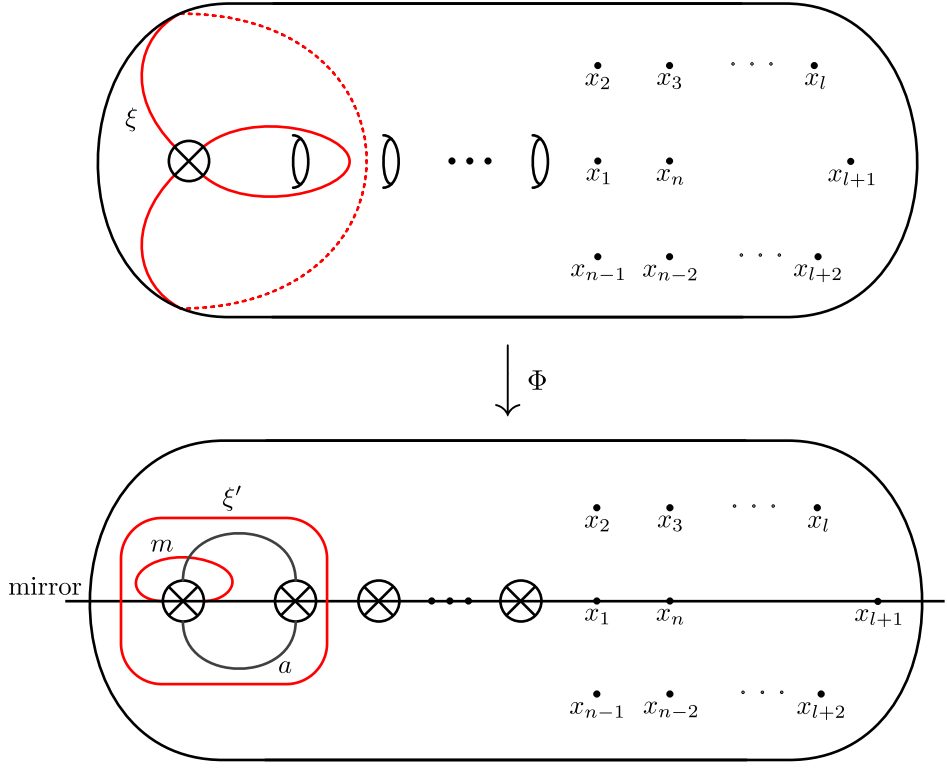}
    \end{center}
    \caption{Diffeomorphism $\Phi$}
    \label{fig:involution_W}
\end{figure}

Let $W$ be a product of $\Phi ^{-1} w \Phi$ and let $\rho _3$ be a product of $W y$. Clearly, we can see that $W$ is an involution. Since we have
\begin{eqnarray*}
    WyW &=& \Phi ^{-1} w  ( \Phi y \Phi ^{-1}) w \Phi \\
    &=& \Phi ^{-1} (w Y_{m,a} w) \Phi \\
    &=& \Phi ^{-1} Y_{m,a}^{-1} \Phi = y^{-1},
\end{eqnarray*}
$\rho_3$ is a involution. So we can get the following lemma.

\begin{lem}
    \label{lem:y_homeo_generate}

    The Y-homeomorphism $y$ is the product of involutions $W$ and $\rho_3$.

\end{lem}

We need the another involution to generate $t_{e_1},t_{e_2},\ldots,t_{e_{n-1}}$. Figure.~\ref{fig:involution_J} gives the involution $J$ which is a reflection in the mirror.

\begin{figure}[ht]
    \begin{center}
        \includegraphics[scale=0.9,clip]{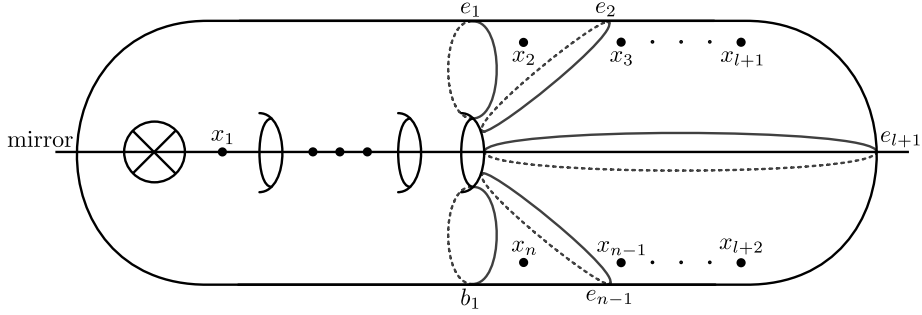}
    \end{center}
    \caption{Involution $J$}
    \label{fig:involution_J}
\end{figure}

\begin{lem}
    \label{lem:dehn_twist_generate2}
    $t_{e_1},t_{e_2},\ldots,t_{e_{n-1}}$ are products of involutions $\sigma,\tau,I,J$ and $\rho_1$.
\end{lem}

\begin{proof}

    Since we have $J(n_1) = e_1$, $t_{e_1} = J t_{n_1} ^{-1} J$.
    $t_{e_1}$ is the product of $\sigma$, $\tau$, $I$, $J$, $\rho_1$.
    Let $T$ denote the product of $JI$. We see that $T$ acts as following.
    \[ (5) T(e_1)=e_2, T(e_2)=e_3,\ldots,T(e_l)=e_{l+1},T(e_{l+1})=e_{l+2},\ldots, T(e_{n-2})=e_{n-1}. \]
    Hence, for $(i=2,3,\ldots,n-1)$, we can see that $t_{e_{i}}$ = $T t_{e_{i-1}} T^{-1}$. So, $t_{e_i}$ is a product of $\sigma,\tau,I,J $ and $\rho_1$.
\end{proof}

Let the subgroup $G$ of $\CN _{g,n}$ be generated by $\sigma,\tau,W,I,J,\rho_1,\rho_2$ and $\rho_3$.

\begin{proof}[Proof of Theorem \ref{thm:main} for genus $g=2r+1$]
    We see that $G$ contains $\PM _{g,n}$ since all Korkmaz's generators for $\PM _{g,n}$ are in $G$ by Lemma~\ref{lem:dehn_twist_generate1},~\ref{lem:puncture_slide_generate},~\ref{lem:y_homeo_generate} and ~\ref{lem:dehn_twist_generate2}.

    When we consider the actions of $\sigma,\tau$ and $W$ on the punctures, we can see that
    \[ \pi (\sigma) = (1,n)(2,n-1)\ldots (l,l+2)(l+1), \]
    \[ \pi (\tau) = (2,n)(3,n-1)\ldots (l+1,l+2)(1), \]
    \[ \pi (W) = (2,n-1)(3,n-2)\ldots (l,l+2)(1)(l+1)(n).\]

    By the following lemma, the restriction $\pi \mid _G : G \to \mathrm{Sym}_n$ is a surjection.
    Hence, we can see that $G = \CN _{g,n}$ by Lemma~\ref{lemma:main_tool_of_proof_main_theorem_two}.
\end{proof}

\begin{lem}
    The group $\mathrm{Sym}_n$ is generated by following elements,
    \[ r_1 = (1,b)(2,n-1)\ldots (l,l+2)(l+1), \]
    \[ r_2 = (2,b)(3,n-1)\ldots (l+1,l+2)(1), \]
    \[ r_3 = (2,n-1)(3,n-2)\ldots (l,l+2)(1)(l+1)(n). \]
\end{lem}

\subsection{In the case of even genus}
In this section, We suppose that $g=2r+2$. Similar to odd case, let us consider the two models of $N_{g,n}$ as shown in Figure.~\ref{fig:reflection_sigma_even} and \ref{fig:reflection_tau_even}. (In these pictures, we will suppose that $r=2k+1$ and the number of punctures $b=2l$ is even.) Each pictures gives an involution of the $N_{g,n}$, which is the reflection in the mirror.

\begin{figure}[ht]
    \begin{center}
        \includegraphics[scale=0.80,clip]{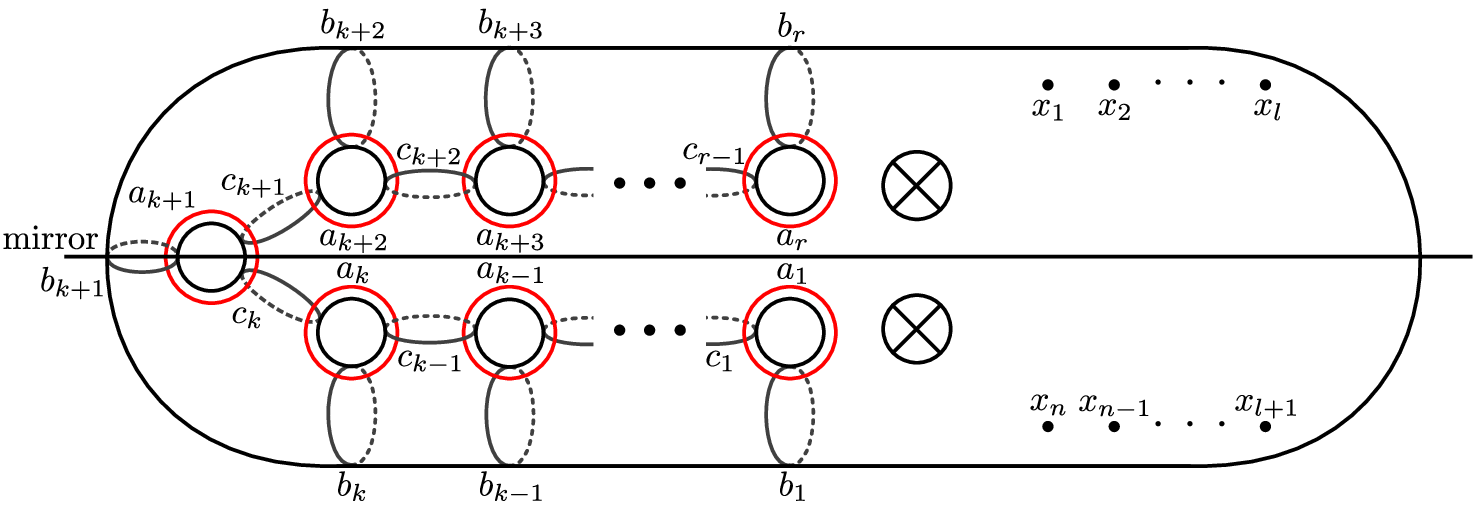}
    \end{center}
    \caption{Involution $\sigma : N_{g,n} \to N_{g,n}$}
    \label{fig:reflection_sigma_even}
\end{figure}

\begin{figure}[ht]
    \begin{center}
        \includegraphics[scale=0.9,clip]{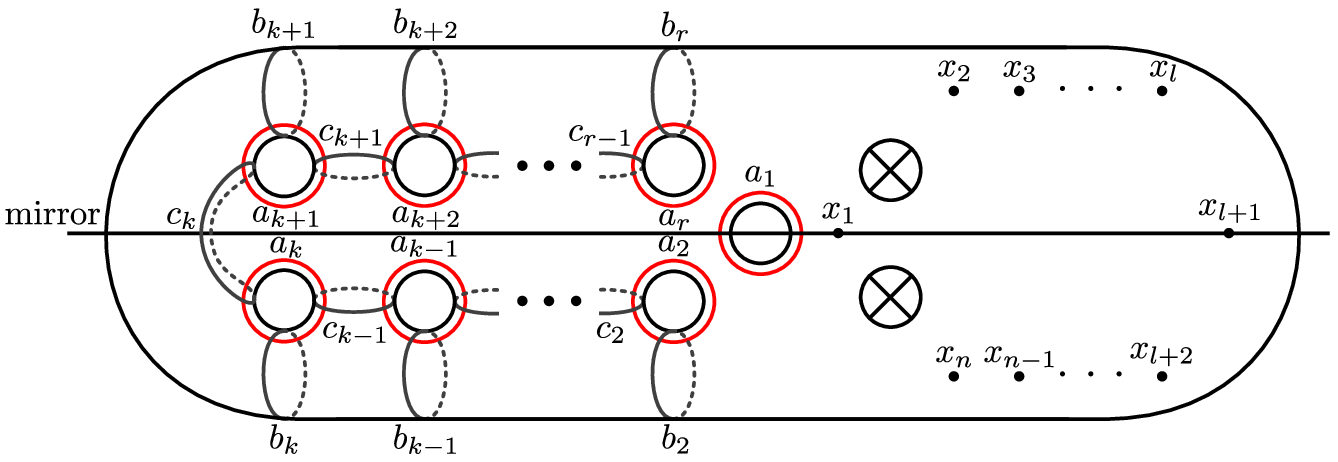}
    \end{center}
    \caption{Involution $\tau : N_{g,n} \to N_{g,n}$}
    \label{fig:reflection_tau_even}
\end{figure}

We will construct third involution $I$. We cut the surface $N_{g,n}$ along $a_{k+3} \cup b_k \cup c_k \cup c_{k+1} \cup x$ to obtain the surfaces $S_1$ and $S_2$.(see Figure.\ref{fig:cut_curves_for_I_even_case}) ~$S_1$ is a sphere bounded by $a_{k+3} \cup b_k \cup c_k \cup c_{k+1} \cup x$ and $S_2$ is a non-orientable surface of genus $g-8$ with $b$ punctures and $5$ boundaries. Figure.\ref{fig:Involution_I_even_case} gives the involutions $\overline{I}$ and $\widetilde{I}$ on $S_1$ and $S_2$, respectively. Since $\overline{I}$ and $\widetilde{I}$ coincide on the boundaries, they naturally define an involution $I:N_{g,n} \to N_{g,n}$.

\begin{figure}[ht]
    \begin{center}
        \includegraphics[scale=1.0,clip]{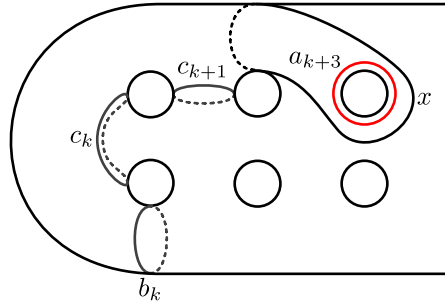}
    \end{center}
    \caption{The curves $a_{k+3}, b_k, c_k, c_{k+1}$ and $x$}
    \label{fig:cut_curves_for_I_even_case}
\end{figure}

\begin{figure}[ht]
    \begin{center}
        \includegraphics[scale=0.9,clip]{involution_i_bar_odd_case}
        \hspace{0.5cm}
        \includegraphics[scale=0.9,clip]{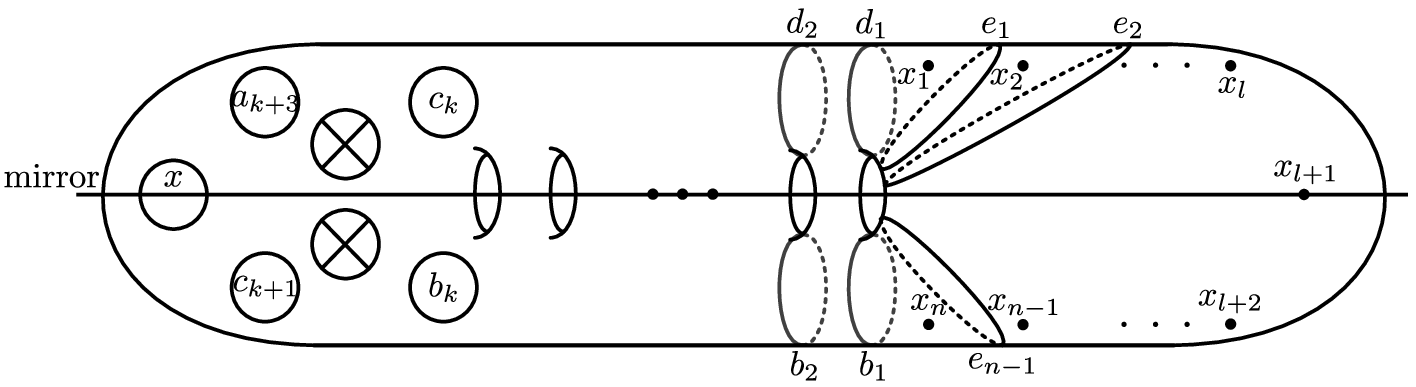}
    \end{center}
    \caption{involutions $\overline{I}$ and $\widetilde{I}$}
    \label{fig:Involution_I_even_case}
\end{figure}

From the construction of $I$, we see the following:
\[ I (a_{k+3}) = c_{k+1} , I(c_k) = b_k, \]
\[ I (b_1) = d_1 , I (b_2) = d_2. \]

Let $\rho _1$ be the product $\tau t_{a_1}$. As in the odd genus case, $\rho _1$ is an involution.
We will prepare three involutions to prove following Lemma. Figure~.\ref{fig:involution_J_even} gives the involution $J$ which is a reflection in the mirror. Let $\rho _4$ and $\rho_5$ be the products $J t_{b_{r+1}}$ and $J t_{c_{r}}$, respectively. We can found that $\rho _4$ and $\rho_5$ are involutions.

\begin{figure}[ht]
    \begin{center}
        \includegraphics[scale=0.9,clip]{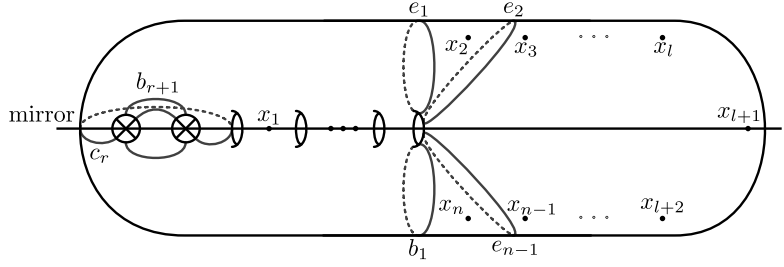}
    \end{center}
    \caption{Involution $J$}
    \label{fig:involution_J_even}
\end{figure}

\begin{lem}
    \label{lem:dehn_twist_generate_even_case}
    Dehn twists $t_{a_1}$, $t_{a_2}$, $\ldots$, $t_{a_r}$, $t_{b_1}$, $t_{b_2}$, $t_{b_{r+1}}$, $t_{c_1}$, $t_{c_2}$, $\ldots$, $t_{c_{r}}$, $t_{d_1}$, $t_{d_2}$, $t_{e_1}$, $t_{e_2}$, $\ldots$, $t_{e_{n-1}}$ are products of involutions $\sigma$, $\tau$, $\rho_1$, $\rho_4$, $\rho_5$, $I$, and $J$.
\end{lem}

\begin{proof}
    Let R be the product $\tau \sigma$. We can see that $R$ acts as following by Figure.~\ref{fig:reflection_sigma_even} and Figure.~\ref{fig:reflection_tau_even}.
    \[ (1) R(a_1) = a_2, R(a_2) = a_3, \ldots, R(a_k)=a_{k+1}, R(a_{k+1}) = a_{k+2}, \ldots, R(a_{r-1}) = a_r.\]
    \[ (2) R(b_1) = b_2, R(b_2) = b_3, \ldots, R(b_k)=b_{k+1}, R(b_{k+1}) = b_{k+2}, \ldots, R(b_{r-1}) = b_r.\]
    \[ (3) R(c_1) = c_2, R(c_2) = c_3, \ldots, R(c_k)=c_{k+1}, R(c_{k+1}) = c_{k+2}, \ldots, R(c_{r-2}) = c_{r-1}. \]
    Clearly, we can see that $t_{a_1}$ is a product of $\tau$ and $\rho_1$. By $(1)$ and Lemma~\ref{lem:Property_Dehn_Twist},
    \[ t_{a_i} = R t_{a_{i-1}} R^{-1}. (i=2,3,\ldots,r) \]
    So $t_{a_1},t_{a_2},\ldots,t_{a_r}$ are products of $\sigma,\tau$, and $\rho_1$.\\

    By construction of $I$ and Lemma~\ref{lem:Property_Dehn_Twist}, we have
    \[ t_{c_{k+1}} = I t_{a_{k+3}} ^{-1} I.\]
    By $(3)$ and Lemma \ref{lem:Property_Dehn_Twist}, we see that
    \[ t_{c_j} = R t_{c_{j-1}} R^{-1}, \hspace{4pt} (j=k+2,k+3,\ldots,r-1) \]
    \[ t_{c_j} = R^{-1} t_{c_{j+1}} R. \hspace{4pt} (j=1,2,\ldots,k) \]
    Hence, $t_{c_1},t_{c_2},\ldots,t_{c_{r-1}}$ are products of $\sigma,\tau,\rho_1$, and $I$.\\

    Also, we have
    \[ t_{{b}_k} = I t_{c_k}^{-1} I\].
    Similar to the above, by $(2)$ and Lemma \ref{lem:Property_Dehn_Twist}, we see that
    \[ t_{b_i} = R t_{b_{i-1}} R^{-1}, \hspace{4pt} (i=k+1,k+3,\ldots,r) \]
    \[ t_{b_i} = R^{-1} t_{b_{i+1}} R. \hspace{4pt} (i=1,2,\ldots,k-1) \]
    Hence, $t_{b_1},t_{b_2},\ldots,t_{b_r}$ are product of $\sigma,\tau,\rho_1$, and $I$.\\

    By the constructions about $\rho_4$ and $\rho_5$,we have $t_{b_{r+1}} = J\rho_4$ and $t_{c_r} = J \rho_5$.

    Since $I(b_1) = d_1$ and $I(b_2) = d_2$, we have
    \[ t_{d_1} = I t_{b_1} ^{-1} I , t_{d_2} = I t_{b_2} ^{-1} I.\]
    $t_{d_1}$ and $t_{d_2}$ are products of $\sigma,\tau,\rho_1$, and $I$.
\end{proof}

We want to generate puncture slides $v_1,v_2,\ldots,v_n$ and $w_1,w_2,\ldots,w_n$ by involutions.
we will construct an involution $K$ which fixes $\alpha_1$ and reverses the orientation of $\alpha_1$. The involution $K$ is a reflection in the mirror in Figure.~\ref{fig:involution_k_even}. Let $\rho_2$ be the product $Kv_1$.

\begin{figure}[ht]
    \begin{center}
        \includegraphics[scale=1.0,clip]{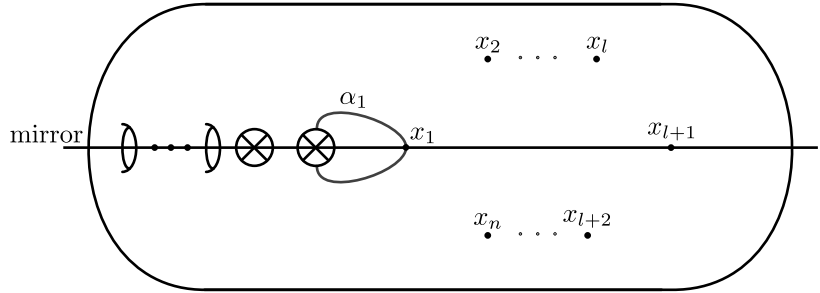}
    \end{center}
    \caption{Involution $K$}
    \label{fig:involution_k_even}
\end{figure}

\begin{lem}
    \label{lem:puncture_slide_generate_even_case}
    Puncture slides $v_i$ and $w_i(i=1,2,\ldots,n)$ are products of involutions $\sigma,\tau$,$K$ and $\rho_2$.
\end{lem}

\begin{proof}
    Since $v_1$ is equal to $K\rho_2$, we can write $v_1$ as a product of two involutions. Let $S$ and $R$ be products $\tau \sigma$ and $\sigma \tau$, respectively. By the constructions of $\sigma$ and $\tau$, we have
    \[S(\alpha_1) = \alpha_2, S(\alpha_2) = \alpha_3, \ldots, S(\alpha_{n-1}) = \alpha_{n}, \]
    \[ R(\beta_n) = \beta_{n-1}, R(\beta_{n-1}) = \beta_{n-2}, \ldots, R(\beta_2) = \beta_1, \]
    \[ \sigma(\alpha_1) = \beta_n. \]

    By lemma~\ref{lemma:Property_puncture_slide}, we can prove this lemma.
\end{proof}

We will write $y$ as a product of involutions. We consider the diffeomorphism $\Phi : N_{g,n} \rightarrow N_{g,n}$ which  satisfies $\Phi y \Phi^{-1} = Y_{m,a}$ and fixes each punctures as shown Figure.~\ref{fig:diffeomorphism_Phi}.

\begin{figure}[ht]
    \begin{center}
        \includegraphics[scale=1.0,clip]{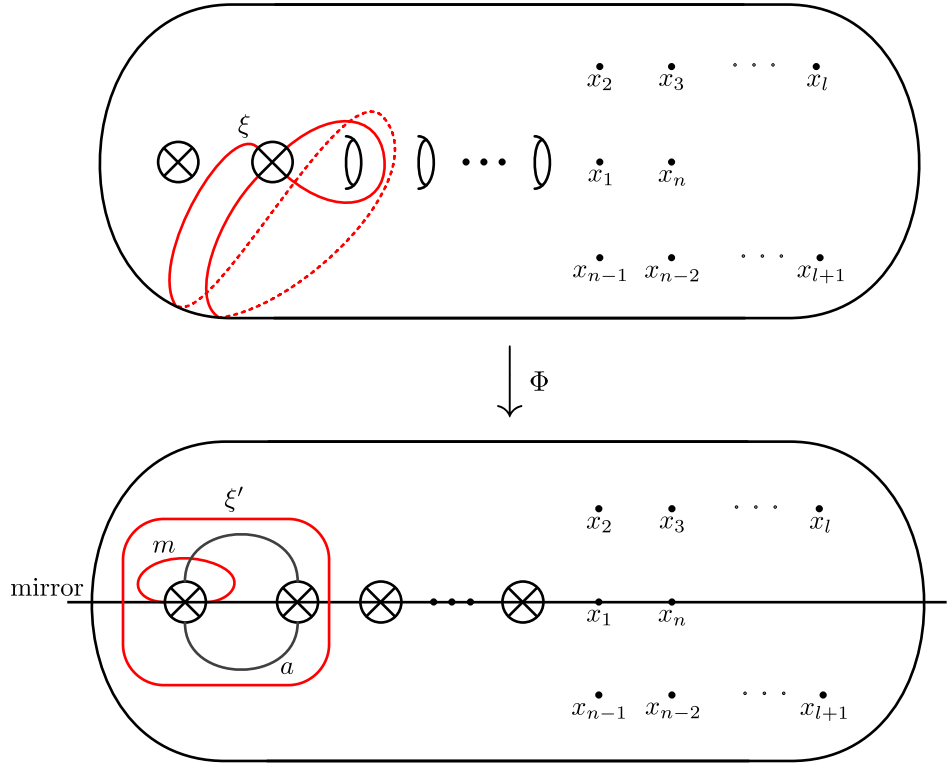}
    \end{center}
    \caption{Diffeomorphism $\Phi$}
    \label{fig:diffeomorphism_Phi}
\end{figure}

Let $\omega$ be reflection in the mirror as shown bottom figure in Figure.~\ref{fig:diffeomorphism_Phi}. Since $\omega$ fixes $m$ and $a$ but reverses the orientation of $m$ and $a$, we can see that $\omega Y_{m,a} \omega = Y_{m,a}^{-1}$. Let $W$ be the product $\Phi^{-1} \omega \Phi$ and let $\rho_3$ be the product $Wy$. We can see that $W$ and $\rho_3$ are involutions.
We can see the following lemma.

\begin{lem}
    \label{lem:y_homeo_generate_even_case}
    The Y-homeomorphism $y$ is the product of involutions $W$ and $\rho_3$.
\end{lem}

Let $G$ be the subgroup of $\CN _{g,n}$ generated by $\sigma,\tau,W,I,J,K,\rho_1,\rho_2,\rho_3,\rho_4$ and $\rho_5$.

\begin{proof}[Proof of Theorem \ref{thm:main} for genus $g=2r+2$]
    We see that $G$ contains $\PM _{g,n}$ since all Korkmaz's generators for $\PM _{g,n}$ are in $G$ by Lemma~\ref{lem:dehn_twist_generate_even_case},~\ref{lem:puncture_slide_generate_even_case} and ~\ref{lem:y_homeo_generate_even_case}.

    When we consider the actions of $\sigma,\tau$ and $W$ on the punctures, we can see that
    \[ \pi (\sigma) = (1,n)(2,n-1)\ldots (l,l+1), \]
    \[ \pi (\tau) = (2,n)(3,n-1)\ldots (l,l+2)(1)(l+1), \]
    \[ \pi (W) = (2,n-1)(3,n-2)\ldots (l,l+1)(1)(n).\]

    By the following lemma, the restriction $\pi \mid _G : G \to \mathrm{Sym}_n$ is a surjection.
    Hence, we can see that $G = \CN _{g,n}$ by Lemma~\ref{lemma:main_tool_of_proof_main_theorem_two}.
\end{proof}

\begin{lem}
    The group $\mathrm{Sym}_n$ is generated by following elements,
    \[ r_1 = (1,n)(2,n-1)\ldots (l,l+1), \]
    \[ r_2 = (2,n)(3,n-1)\ldots (l,l+2)(1)(l+1), \]
    \[ r_3 = (2,n-1)(3,n-2)\ldots (l,l+1)(1)(n). \]
\end{lem}


\section*{Acknowledgments}
The author would like to thank Susumu Hirose, Noriyuki Hamada, and Naoyuki Monden for helpful comments and invaluable advice on the mapping class groups of surfaces. He would also like to thank Professor Osamu Saeki for many helpful suggestions and comments.



\end{document}